\documentclass{article}

\usepackage{amsmath,amssymb,amsthm}
\usepackage{tikz}
    \usetikzlibrary{decorations.pathreplacing}
\usepackage{xcolor}
\usepackage{float}
\usepackage{array}
\usepackage{tabu}
\usepackage{amsfonts}
\usepackage{bold-extra}
\usepackage{multirow}
\usepackage{pdfpages}

\definecolor{webgreen}{rgb}{0,.5,0}
\definecolor{webbrown}{rgb}{.6,0,0}
\usepackage[colorlinks=true,
linkcolor=webgreen,
filecolor=webbrown,
citecolor=webgreen]{hyperref}
\usepackage[capitalise,noabbrev]{cleveref}
\usepackage{longtable}
\usepackage{listings}
\definecolor{green}{rgb}{0.0, 0.5, 0.0}
\definecolor{comment}{rgb}{0.0, 0.5, 0.5}
\definecolor{string}{rgb}{0.73, 0.14, 0.14}

\lstdefinestyle{mystyle}{
    language=Python,
    keywordstyle=\color{green},   
    commentstyle=\color{comment},
    stringstyle = \color{string},
    numberstyle= \color{green},
    basicstyle=\ttfamily\footnotesize,
    breakatwhitespace=false,         
    breaklines=true,                 
    captionpos=b,                    
    keepspaces=true,                 
    showspaces=false,                
    showstringspaces=false,
    showtabs=false,                  
    tabsize=2
}

\theoremstyle{definition}
\newtheorem{definition}{Definition}[section]

\theoremstyle{plain}

\newtheorem{lemma}[definition]{Lemma}

\newtheorem{proposition}[definition]{Proposition}

\theoremstyle{remark}

\newcommand{\tile}[2]{\fill[fill=gray] (#1,#2) rectangle (#1 + 1,#2 + 1);}

\title{High Temperature domineering Positions}
\author{}
\date{}

\begin{document}

\begin{center}
  \vskip 1cm{\LARGE\bf
      High Temperature \textsc{Domineering} Positions
    }
    \vskip 1cm
  Svenja Huntemann\footnote{This author's research is supported in part by the Natural Sciences and Engineering Research Council of Canada grant 2022-04273.
  \textit{MSC subject class:} 91A46}\\
  Department of Mathematics and Statistics\\
  Mount Saint Vincent University\\
  Halifax, NS~~B3M~2J6\\
  Canada\\
  \href{mailto:svenja.huntemann@msvu.ca}{\tt svenja.huntemann@msvu.ca} \\
  \ \\
  Tomasz Maciosowski \\
  %Department of Mathematics and Physical Sciences\\
  Concordia University of Edmonton\\
  %Edmonton, AB~~T5B~4E4\\
  Canada\\
  \href{mailto:t.maciosowski@gmail.com}{\tt t.maciosowski@gmail.com} \\
\end{center}

\vskip .2 in

\begin{abstract}
  \textsc{Domineering} is a partizan game where two players have a collection of dominoes which they place on the grid in turn, covering up squares.
  One player places tiles vertically, while the other places them horizontally; the first player who cannot move loses.
  It has been conjectured that the highest temperature possible in \textsc{Domineering} is 2.
  We have developed a program that enables a parallel exhaustive search of \textsc{Domineering} positions with temperatures close to or equal to 2 to allow for analysis of such positions.
\end{abstract}

\section{Introduction}

\textsc{Domineering} is a partizan game where Left places blue, vertical dominoes, and Right places red, horizontal dominoes.
We will assume that if a player cannot make a move on their turn, then they lose.

Traditionally, \textsc{Domineering} is played on a rectangular board. As a game progresses, it naturally breaks into smaller components though that can have different shapes (see \cref{fig:DomExampleGame}). On their turn, the players now choose a component and make their move in it. The decomposition into such components is called the \textit{disjunctive sum}.

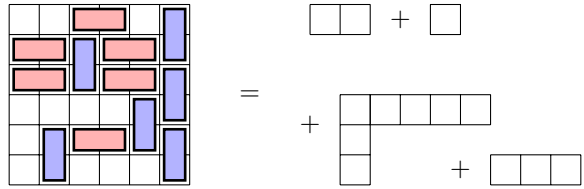
\begin{figure}[H]
  \centering
  % From Svenja Huntemann
  \begin{tikzpicture}[scale=0.4]
    \draw (0,0) grid (6, 6);
    \filldraw[fill=red!30, line width=1pt] (0.15,4.15)--(1.85,4.15)--(1.85,4.85)--(0.15,4.85)-- cycle;
    \filldraw[fill=blue!30, line width=1pt] (1.15,0.15)--(1.85,0.15)--(1.85,1.85)--(1.15,1.85)-- cycle;
    \filldraw[fill=red!30, line width=1pt] (0.15,3.15)--(1.85,3.15)--(1.85,3.85)--(0.15,3.85)-- cycle;
    \filldraw[fill=blue!30, line width=1pt] (2.15,3.15)--(2.85,3.15)--(2.85,4.85)--(2.15,4.85)-- cycle;
    \filldraw[fill=red!30, line width=1pt] (3.15,3.15)--(4.85,3.15)--(4.85,3.85)--(3.15,3.85)-- cycle;
    \filldraw[fill=blue!30, line width=1pt] (4.15,1.15)--(4.85,1.15)--(4.85,2.85)--(4.15,2.85)-- cycle;
    \filldraw[fill=blue!30, line width=1pt] (5.15,0.15)--(5.85,0.15)--(5.85,1.85)--(5.15,1.85)-- cycle;
    \filldraw[fill=blue!30, line width=1pt] (5.15,2.15)--(5.85,2.15)--(5.85,3.85)--(5.15,3.85)-- cycle;
    \filldraw[fill=blue!30, line width=1pt] (5.15,4.15)--(5.85,4.15)--(5.85,5.85)--(5.15,5.85)-- cycle;
    \filldraw[fill=red!30, line width=1pt] (3.15,4.15)--(4.85,4.15)--(4.85,4.85)--(3.15,4.85)-- cycle;
    \filldraw[fill=red!30, line width=1pt] (2.15,1.15)--(3.85,1.15)--(3.85,1.85)--(2.15,1.85)-- cycle;
    \filldraw[fill=red!30, line width=1pt] (2.15,5.15)--(3.85,5.15)--(3.85,5.85)--(2.15,5.85)-- cycle;
    \node at (8,3) {$=$};
    \begin{scope}[shift={(10,5)}]
        \draw (0,0) grid (2, 1);
        \node at (3,0.5) {$+$};
        \draw (4,0) grid (5,1);
    \end{scope}
    \begin{scope}[shift={(11,0)}]
        \draw (0,0) grid (1,3);
        \draw (1,2) grid (5,3);
    \end{scope}
    \begin{scope}[shift={(16,0)}]
        \draw (0,0) grid (3,1);
    \end{scope}
    \node at (10,2) {$+$};
    \node at (15,0.5) {$+$};
  \end{tikzpicture}
  \caption[]{A possible position in \textsc{Domineering}}
  \label{fig:DomExampleGame}
\end{figure}

We consider each of the components as a game on their own, with possibly non-alternating play. To determine which player wins in a sum and how, we can assign to each game a \textit{game value} which indicates which player wins in it and how much of an advantage they have. 

Similarly, we can assign to each component a numerical value, called the \textit{temperature}, that indicates the urgency of moving in it, or, equivalently, how much the first player to go gains by moving in it. Note however that the best move is not always in the hottest component.

For more details about combinatorial game theory, the study of games similar to \textsc{Domineering}, including the formal definitions of value and temperature, see for example \cite{BerlekampCG2004,siegel2013combinatorial}.

An open question is what temperatures are possible in \textsc{Domineering} (see for example \cite{BerlekampCG2004,guy1995unsolved,Berlekamp2019}) and it has been conjectured that the maximum possible temperature is 2.

%Berlekamp conjectured in the 1970s {\color{red} check year} that any \textsc{Domineering} position has a temperature of at most $2$ (see for example \cite{drummond2004temperature}). This conjecture is still open despite extensive work by several groups. 

Drummond-Cole has found the first \textsc{Domineering} position with temperature $2$ \cite{drummond2004temperature}, which we show below. %higher than any of the positions in a table from ``Unsolved problems in combinatorial games'' \cite{guy1995unsolved}

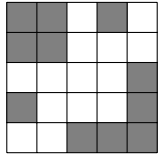
\begin{center}
    \begin{tikzpicture}[scale=0.4] \fill[fill=gray] (2,0) rectangle (3,1); \fill[fill=gray] (3,0) rectangle (4,1); \fill[fill=gray] (4,0) rectangle (5,1); \fill[fill=gray] (0,1) rectangle (1,2); \fill[fill=gray] (4,1) rectangle (5,2); \fill[fill=gray] (4,2) rectangle (5,3); \fill[fill=gray] (0,3) rectangle (1,4); \fill[fill=gray] (1,3) rectangle (2,4); \fill[fill=gray] (0,4) rectangle (1,5); \fill[fill=gray] (1,4) rectangle (2,5); \fill[fill=gray] (3,4) rectangle (4,5); \draw[step=1cm,black] (0,0) grid (5, 5); \end{tikzpicture}
\end{center}

Shankar and Sridharan in 2005 in ``New temperatures in Domineering'' list \textsc{domineering} positions with various temperatures found from a search of $5 \times 6$, $4 \times 8$, $3 \times 10$, and $2 \times 16$ grids \cite{shankar2005new}, including temperatures close to 2.

At a Virtual CGT Workshop in 2020, several more positions with temperature 2 were found by trial and error. All positions of temperature $2$ and close to $2$ found thus far have a common structure, which we call a hook, see below.

%\begin{figure}[!ht]
    
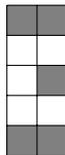
\begin{center}
    \begin{tikzpicture}[scale=0.4] \fill[fill=gray] (0,0) rectangle (1,1); \fill[fill=gray] (1,0) rectangle (2,1); \fill[fill=gray] (1,2) rectangle (2,3); \fill[fill=gray] (0,4) rectangle (1,5); \fill[fill=gray] (1,4) rectangle (2,5); \draw[step=1cm,black] (0,0) grid (2, 5); \end{tikzpicture}
    \end{center}
%  \caption[]{Hook}
%  \label{fig:hook}
%\end{figure}

We are interested in whether the hook is a necessary condition for a high-temperature position in \textsc{Domineering} since this might open new avenues towards a proof of Berlekamp's conjecture. Note that the hook is not sufficient. For example, the position shown below has temperature 0, thus quite low.

%\begin{figure}[!ht]
    
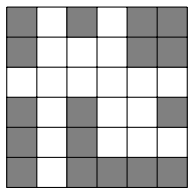
\begin{center} 
    \begin{tikzpicture}[scale=0.4] \fill[fill=gray] (0,0) rectangle (1,1); \fill[fill=gray] (2,0) rectangle (3,1); \fill[fill=gray] (3,0) rectangle (4,1); \fill[fill=gray] (4,0) rectangle (5,1); \fill[fill=gray] (5,0) rectangle (6,1); \fill[fill=gray] (0,1) rectangle (1,2); \fill[fill=gray] (2,1) rectangle (3,2); \fill[fill=gray] (0,2) rectangle (1,3); \fill[fill=gray] (2,2) rectangle (3,3); \fill[fill=gray] (5,2) rectangle (6,3); \fill[fill=gray] (0,4) rectangle (1,5); \fill[fill=gray] (4,4) rectangle (5,5); \fill[fill=gray] (5,4) rectangle (6,5); \fill[fill=gray] (0,5) rectangle (1,6); \fill[fill=gray] (2,5) rectangle (3,6); \fill[fill=gray] (4,5) rectangle (5,6); \fill[fill=gray] (5,5) rectangle (6,6); \draw[step=1cm,black] (0,0) grid (6, 6); \end{tikzpicture}
    \end{center}
%  \caption[]{Position with a hook and temperature $t=0$}
%  \label{fig:lowTempHook}
%\end{figure}

\section{Search}

As a first step towards identifying whether the hook is necessary for a high temperature position, we started an exhaustive search for positions with temperature above a given threshold.

We developed a program that allows for parallel, exhaustive search of \textsc{Domineering} positions to compute game values, thermographs (used to find temperature), and temperatures.
The implementation approach has been heavily inspired by Aaron Siegel's cgsuite \cite{CGSuite} but implemented from scratch in Rust for improved performance.

For easier data storage, we will think of each game as embedded into a rectangle, with some spaces already occupied. For example, the game below indicates a disjunctive sum of two games, from which we can then remove several completely occupied rows and columns to reach the smallest possible rectangles in which to embed the components.

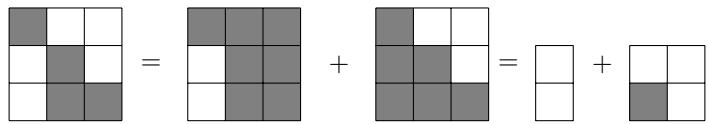
\begin{center}
  \begin{tikzpicture}
    \begin{scope}[scale=0.5] \fill[fill=gray] (1,0) rectangle (2,1); \fill[fill=gray] (2,0) rectangle (3,1); \fill[fill=gray] (1,1) rectangle (2,2); \fill[fill=gray] (0,2) rectangle (1,3); \draw[step=1cm,black] (0,0) grid (3, 3); \end{scope}
    \node[text width=1cm, align=left] at (2.25,0.75) {=};
    \begin{scope}[scale=0.5, shift={(4.75,0)}] \fill[fill=gray] (1,0) rectangle (2,1); \fill[fill=gray] (2,0) rectangle (3,1); \fill[fill=gray] (1,1) rectangle (2,2); \fill[fill=gray] (2,1) rectangle (3,2); \fill[fill=gray] (0,2) rectangle (1,3); \fill[fill=gray] (1,2) rectangle (2,3); \fill[fill=gray] (2,2) rectangle (3,3); \draw[step=1cm,black] (0,0) grid (3, 3); \end{scope}
    \node[text width=1cm, align=left] at (4.75,0.75) {+};
    \begin{scope}[scale=0.5, shift={(9.75,0)}] \fill[fill=gray] (0,0) rectangle (1,1); \fill[fill=gray] (1,0) rectangle (2,1); \fill[fill=gray] (2,0) rectangle (3,1); \fill[fill=gray] (0,1) rectangle (1,2); \fill[fill=gray] (1,1) rectangle (2,2); \fill[fill=gray] (0,2) rectangle (1,3); \draw[step=1cm,black] (0,0) grid (3, 3); \end{scope}
    \node[text width=1cm, align=left] at (7,0.75) {=};
    \begin{scope}[scale=0.5, shift={(14,0)}] \draw[step=1cm,black] (0,0) grid (1, 2); \end{scope}
    \node[text width=1cm, align=left] at (8.25,0.75) {+};
    \begin{scope}[scale=0.5, shift={(16.5,0)}] \fill[fill=gray] (0,0) rectangle (1,1); \draw[step=1cm,black] (0,0) grid (2, 2); \end{scope}
  \end{tikzpicture}
\end{center}

It is known that the temperature of a disjoint sum of games is no larger than the temperature of any one of the components, i.e.
\[t(G + H) \leq \max \left\{ t(G), t(H) \right\}\] for any two games $G$ and $H$ (see for example \cite{siegel2013combinatorial}). Due to this, in the search for high temperature positions it is enough if we consider only positions that do not decompose into a disjunctive sum.

Also note that we did not only consider positions that can be reached during actual game play from a rectangle, but any possible subgrid of rectangular grids.

\section{Infinite Family of Positions With Temperature \texorpdfstring{$2$}{}}\label{sec:familytemp2}

During our search we have found an infinite family of positions which have temperature 2. 
A position in this family is built by appending $2 \times 2$ $L$-shapes to the bottom of the Drummond-Cole position. We will denote the position with $n$ such $L$-shapes added by $DCL_n$, see \cref{fig:DCL}.

\begin{figure}[H]
  \centering
  \begin{tikzpicture}[scale=0.4]
    \begin{scope}
      \tile{0}{8};
      \tile{1}{8};
      \tile{3}{8};
      \tile{0}{7};
      \tile{1}{7};
      \tile{4}{6};
      \tile{0}{5};
      \tile{4}{5};
      \tile{2}{4};
      \tile{0}{3};
      \tile{1}{3};
      \tile{2}{3};
      \tile{4}{3};
      \tile{0}{2};
      \tile{1}{2};
      \tile{2}{2};
      \tile{0}{1};
      \tile{1}{1};
      \tile{2}{1};
      \tile{4}{1};
      \tile{0}{0};
      \tile{1}{0};
      \tile{2}{0};
      \tile{0}{-1};
      \tile{1}{-1};
      \tile{2}{-1};
      \tile{4}{-1};
      \draw[step=1cm,black] (0,-1) grid (5, 9);
    \end{scope}

    \node[] at (2.5, -1.5) {$\vdots$};

    \begin{scope}[shift={(0, -4.5)}]
      \tile{0}{1};
      \tile{1}{1};
      \tile{2}{1};
      \tile{0}{0};
      \tile{1}{0};
      \tile{2}{0};
      \tile{4}{0};
      \draw[step=1cm,black] (0,0) grid (5, 2);
    \end{scope}

    \draw[decorate, decoration={brace,amplitude=5pt,mirror,raise=4ex}] (4.25,-4.5)--(4.25,5) node[midway,xshift=3em] {$n$};
  \end{tikzpicture}
  \caption[]{The position $DCL_n$, which has temperature $2$}
  \label{fig:DCL}
\end{figure}
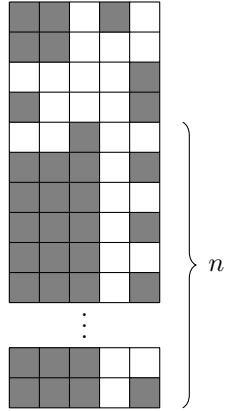

We claim that this position is a disjunctive sum of the Drummond-Cole position and the chain of $L$s. The former has value $\pm(2*)$, while the chain has value $0$ or $*$. Thus the overall position has canonical form $\pm (2*)$ or $\pm 2$ depending on $n$. To show this, we will prove the values of several related positions first. We will denote by $L_n$ a chain of $n$ of the $2\times 2$ $L$-shapes. Then $L_n^+$ has one additional square in the path, while $L_n^-$ has one removed. Finally, $L_n^\cup$ is $L_n^+$ with an additional square and a hook on one end added. For all four types of positions, see \cref{fig:Lns}.

\begin{figure}[!ht]
    \centering
    \begin{tikzpicture}[scale=0.4]
        \tile{1}{0}
        \draw[step=1cm,black] (0,0) grid (2,2);
        \node at (1, 3.25) {$\vdots$};
        \tile{1}{4}
        \tile{1}{6}
        \draw[step=1cm,black] (0,4) grid (2,8);
        \draw[decorate, decoration={brace,amplitude=5pt,mirror,raise=4ex}] (1.25,0)--(1.25,8) node[midway,xshift=3em] {$n$};

        \begin{scope}[shift={(8, 0)}]
        \tile{1}{0}
        \draw[step=1cm,black] (0,0) grid (2,2);
        \node at (1, 3) {$\vdots$};
        \tile{1}{4}
        \tile{1}{6}
        \tile{1}{8}
        \draw[step=1cm,black] (0,4) grid (2,9);
        \draw[decorate, decoration={brace,amplitude=5pt,mirror,raise=4ex}] (1.25,0)--(1.25,8) node[midway,xshift=3em] {$n$};
        \end{scope}

        \begin{scope}[shift={(16, 0)}]
        \draw[step=1cm,black] (0,1) grid (2,2);
        \node at (1, 3) {$\vdots$};
        \tile{1}{4}
        \tile{1}{6}
        \draw[step=1cm,black] (0,4) grid (2,8);
        \draw[decorate, decoration={brace,amplitude=5pt,mirror,raise=4ex}] (1.25,2)--(1.25,8) node[midway,xshift=3.75em] {$n-1$};
        \end{scope}

        \begin{scope}[shift={(25, 0)}]
        \tile{1}{0}
        \tile{-1}{0}
        \tile{-1}{1}
        \draw[step=1cm,black] (-1,0) grid (2,2);
        \node at (0.5, 3) {$\vdots$};
        \tile{1}{4}
        \tile{1}{6}
        \tile{1}{8}
        \tile{1}{9}
        \foreach \y in {4,5,...,9} {\tile{-1}{\y}}
        \tile{0}{11}
        \draw[step=1cm,black] (-1,4) grid (2,12);
        \draw[decorate, decoration={brace,amplitude=5pt,mirror,raise=4ex}] (1.25,0)--(1.25,8) node[midway,xshift=3em] {$n$};
        \end{scope}
    \end{tikzpicture}
    \caption{The positions $L_n$ (left), $L_n^+$ (centre left), $L_n^-$ (centre right), and $L_n^\cup$ (right)}
    \label{fig:Lns}
\end{figure}
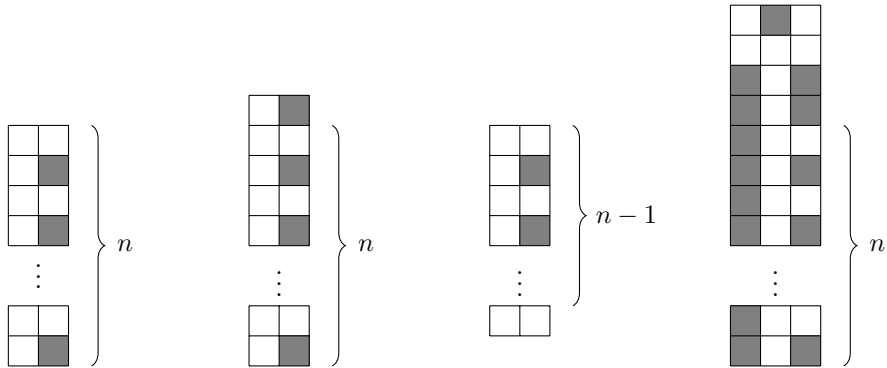

\begin{lemma}
The positions $L_n$, $L_n^+$, and $L_n^-$ have value $0$ if $n$ is even and $*$ if $n$ is odd.
\end{lemma}
We technically do not need all three of these families for the proof of the value of $DCL_n$, but since they are options of each other, it is easier to prove them all at once.
\begin{proof}
Let $c(n)=0$ when $n$ is even and $c(n)=*$ when $n$ is odd. We will show that $L_n=L_n^+=L_n^-=c(n)$ by strong induction.

For the base cases, $L_1$ and $L_1^+$ have value $*$ since for both players the only moves are to two single spaces, which is 0. Now assume that the result holds for $n\leq k-1$. Since a single square has value 0, we will ignore them in the options.

For $L_n$ any move by Left results in the sum $L_i+L_j$ or $L_i^-+L_j^+$ with $i+j=n-1$, or goes to $L_{n-1}$. By induction, all of these have value $c(n-1)$. Any Right move is to a sum $L_i+L_j^+$ with $i+j=n-1$, or to $L_{n-1}^+$. Again, all have value $c(n-1)$. Thus $L_n$ has value $c(n)$.

For $L_n^+$ any move by Left results in the sum $L_i^++L_j$ with $i+j=n-1$, or goes to $L_{n-1}$ or $L_{n-1}^+$. By induction, all of these have value $c(n-1)$. Any Right move is to a sum $L_i^++L_j^+$ with $i+j=n-1$, or to $L_{n-1}^+$. Again, all have value $c(n-1)$. Thus $L_n^+$ has value $c(n)$.

For $L_n^-$ any move by Left results in the sum $L_i^-+L_j$ with $i+j=n-1$, or goes to $L_{n-1}^-$. By induction, all of these have value $c(n-1)$. Any Right move is to a sum $L_i+L_j$ with $i+j=n-1$, or to $L_{n-1}$. Again, all have value $c(n-1)$. Thus $L_n^-$ has value $c(n)$.
\end{proof}

\begin{lemma}
    The position $L_n^\cup$ has value $2*$ if $n$ is even and $2$ if $n$ is odd.
\end{lemma}
\begin{proof}
    We will show that $L_n^\cup$ is equal to the disjunctive sum of the hook with two squares added and $L_n$, i.e.\ that
    
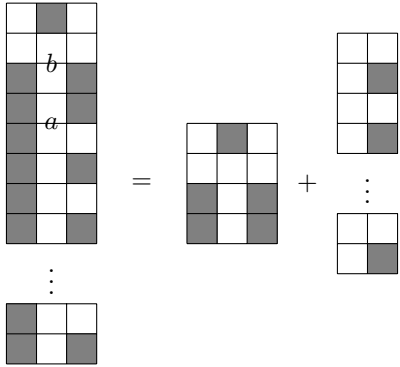
\begin{center}
        \begin{tikzpicture}[scale=0.4]
        \tile{1}{0}
        \tile{-1}{0}
        \tile{-1}{1}
        \draw[step=1cm,black] (-1,0) grid (2,2);
        \node at (0.5, 3) {$\vdots$};
        \tile{1}{4}
        \tile{1}{6}
        \tile{1}{8}
        \tile{1}{9}
        \foreach \y in {4,5,...,9} {\tile{-1}{\y}}
        \tile{0}{11}
        \draw[step=1cm,black] (-1,4) grid (2,12);

        \node at (0.5,8) {$a$};
        \node at (0.5,10) {$b$};
        
        \node at (3.5,6) {$=$};
                
        \begin{scope}[shift={(6, 4)}]
        \tile{1}{0}
        \tile{1}{1}
        \tile{-1}{0}
        \tile{-1}{1}
        \tile{0}{3}
        \draw[step=1cm,black] (-1,0) grid (2,4);
        \end{scope}

        \node at (9,6) {$+$};

        \begin{scope}[shift={(10, 3)}]
        \tile{1}{0}
        \draw[step=1cm,black] (0,0) grid (2,2);
        \node at (1, 3) {$\vdots$};
        \tile{1}{4}
        \tile{1}{6}
        \draw[step=1cm,black] (0,4) grid (2,8);
        \end{scope}
        \end{tikzpicture}
    \end{center}
    We will do so by showing that Left's move at $a$ in $L_n^\cup$ is strictly dominated by the move in $b$.

    When Left moves in $a$, the result is a sum of the hook with a single square attached, an isolated square, and $L_{n-1}^+$. This has value $\{2\mid 1\}+*$ when $n$ is even and $\{2\mid 1\}$ when $n$ is odd.

    On the other hand, when Left moves in $b$, the resulting sum is two vertical dominoes plus $L_n^+$, which has value $2$ when $n$ is even and $2*$ when $n$ is odd. This is strictly better for Left, thus she will never play in $a$ and the sum claimed above holds.

    Since the value of the hook with two squares added is $2*$, the value of $L_n^\cup$ is as claimed.
\end{proof}

We will now prove the value of $DCL_n$ in a similar manner.
\begin{proposition}
    The position $DCL_n$ has value $\pm(2*)$ if $n$ is even and $\pm2$ if $n$ is odd. Thus the temperature is 2.
\end{proposition}

\begin{proof}
    We will show that $DCL_n$ is equal to the disjunctive sum of the Drummond-Cole position and $L_n$, i.e.\ that
    
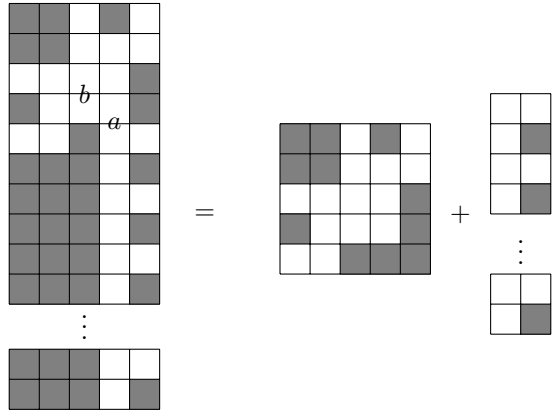
\begin{center}
        \begin{tikzpicture}[scale=0.4]
        \begin{scope}
      \tile{0}{8};
      \tile{1}{8};
      \tile{3}{8};
      \tile{0}{7};
      \tile{1}{7};
      \tile{4}{6};
      \tile{0}{5};
      \tile{4}{5};
      \tile{2}{4};
      \tile{0}{3};
      \tile{1}{3};
      \tile{2}{3};
      \tile{4}{3};
      \tile{0}{2};
      \tile{1}{2};
      \tile{2}{2};
      \tile{0}{1};
      \tile{1}{1};
      \tile{2}{1};
      \tile{4}{1};
      \tile{0}{0};
      \tile{1}{0};
      \tile{2}{0};
      \tile{0}{-1};
      \tile{1}{-1};
      \tile{2}{-1};
      \tile{4}{-1};
      \draw[step=1cm,black] (0,-1) grid (5, 9);
    \end{scope}

    \node[] at (2.5, -1.5) {$\vdots$};

    \begin{scope}[shift={(0, -4.5)}]
      \tile{0}{1};
      \tile{1}{1};
      \tile{2}{1};
      \tile{0}{0};
      \tile{1}{0};
      \tile{2}{0};
      \tile{4}{0};
      \draw[step=1cm,black] (0,0) grid (5, 2);
    \end{scope}

        \node at (3.5,5) {$a$};
        \node at (2.5,6) {$b$};
        
        \node at (6.5,2) {$=$};
                
        \begin{scope}[shift={(9, 0)}]
        \tile{0}{1}
        \tile{0}{3}
        \tile{0}{4}
        \tile{1}{3}
        \tile{1}{4}
        \tile{2}{0}
        \tile{3}{0}
        \tile{3}{4}
        \tile{4}{0}
        \tile{4}{1}
        \tile{4}{2}
        \draw[step=1cm,black] (0,0) grid (5,5);
        \end{scope}

        \node at (15,2) {$+$};

        \begin{scope}[shift={(16, -2)}]
        \tile{1}{0}
        \draw[step=1cm,black] (0,0) grid (2,2);
        \node at (1, 3) {$\vdots$};
        \tile{1}{4}
        \tile{1}{6}
        \draw[step=1cm,black] (0,4) grid (2,8);
        \end{scope}
        \end{tikzpicture}
    \end{center}
    We will do so by showing that Left's move at $a$ in $DCL_n$ is strictly dominated by the move in $b$.

    When Left moves in $a$, the result is a sum of the Drummond-Cole position with the corner square removed (which has value $\pm\{2\mid 1\}$), an isolated square, and $L_{n-1}^+$. This has value $\pm\{2\mid 1\}+*$ when $n$ is even and $\pm\{2\mid 1\}$ when $n$ is odd.

    On the other hand, when Left moves in $b$, the resulting sum is a hook plus $L_n^\cup$, which has value $2*$ when $n$ is even and $2$ when $n$ is odd. This is strictly better for Left, thus she will never play in $a$ and the sum claimed above holds.

    Since the value of the Drummond-Cole position is $\pm(2*)$, the proposition holds.
\end{proof}

% The Drummond-Cole position has value $\pm 2*$, and a single $L$ has value $*$.
% Even when The Drummond-Cole position and the $L$ are adjacent they still can be considered as a disjoint sum as only Left player can move ``between'' them but that would lead to a position with value $\pm \{2 \mid 1\}$ thus that move is dominated by the one that moves to $2*$.
% Two $L$s as well can be considered as a disjoint sum as the move between by Left moves to a position with value $-1$ while a move to $*$ was available.
% Thus the position has canonical form $\pm 2 + (*k)$ where $k$ is the number of added $L$s, and temperature $2$.

\section{Results of Search}

We have developed and run a program that allows for parallel exhaustive search of \textsc{Domineering} positions and computes canonical forms, thermographs, and temperatures.
We have run it for grids of size $5 \times 5$ and $5 \times 6$.
While it is possible to run it for larger grids, the number of positions grows exponentially with the size of the grid and the search space becomes too large to be feasible, mainly hitting the memory limits of the machine to store the transposition table of already computed positions. Without the use of a transposition table, the search would be too slow.
The estimated memory requirement for the exhaustive $6 \times 6$ grid search is about 4TB.

In the tables below, we are omitting positions that are a rotation or reflection of already included positions as they have the same temperature. For the $5\times 5$ and $5\times 6$ searches, we are including only positions with temperatures higher than $7/4$ in the tables.

Notice that all positions listed contain the hook. Further, all temperature $2$ positions contain the Drummond-Cole position -- although we do not prove so, it is likely that most, if not all, of these positions are a disjunctive sum of the Drummond-Cole position with the rest, the latter of which is an infinitesimal or a number, similar to the infinite family $DCL_n$ in \cref{sec:familytemp2}.

Attached is an output file with the computer readable positions and temperatures computed.

% https://tex.stackexchange.com/a/5317
\subsection{\texorpdfstring{$5 \times 5$}{}}

{
%% Auto generated by `cgt-cli`
%% Make sure to include preamble from README.md
\begin{longtabu}{m{2cm} m{1cm}|m{2cm} m{1cm}|m{2cm} m{1cm}} 
\hline Position & Temp. & Position & Temp. & Position & Temp. \\ \hline \endhead
\begin{tikzpicture}[scale=0.4] \fill[fill=gray] (2,0) rectangle (3,1); \fill[fill=gray] (3,0) rectangle (4,1); \fill[fill=gray] (4,0) rectangle (5,1); \fill[fill=gray] (0,1) rectangle (1,2); \fill[fill=gray] (4,1) rectangle (5,2); \fill[fill=gray] (4,2) rectangle (5,3); \fill[fill=gray] (0,3) rectangle (1,4); \fill[fill=gray] (1,3) rectangle (2,4); \fill[fill=gray] (0,4) rectangle (1,5); \fill[fill=gray] (1,4) rectangle (2,5); \fill[fill=gray] (3,4) rectangle (4,5); \draw[step=1cm,black] (0,0) grid (5, 5); \end{tikzpicture} & $2$ & \begin{tikzpicture}[scale=0.4] \fill[fill=gray] (2,0) rectangle (3,1); \fill[fill=gray] (0,1) rectangle (1,2); \fill[fill=gray] (4,1) rectangle (5,2); \fill[fill=gray] (4,2) rectangle (5,3); \fill[fill=gray] (0,3) rectangle (1,4); \fill[fill=gray] (1,3) rectangle (2,4); \fill[fill=gray] (0,4) rectangle (1,5); \fill[fill=gray] (1,4) rectangle (2,5); \fill[fill=gray] (3,4) rectangle (4,5); \draw[step=1cm,black] (0,0) grid (5, 5); \end{tikzpicture} & $15/8$ & \begin{tikzpicture}[scale=0.4] \fill[fill=gray] (2,0) rectangle (3,1); \fill[fill=gray] (3,0) rectangle (4,1); \fill[fill=gray] (0,1) rectangle (1,2); \fill[fill=gray] (4,2) rectangle (5,3); \fill[fill=gray] (0,3) rectangle (1,4); \fill[fill=gray] (1,3) rectangle (2,4); \fill[fill=gray] (0,4) rectangle (1,5); \fill[fill=gray] (1,4) rectangle (2,5); \fill[fill=gray] (3,4) rectangle (4,5); \draw[step=1cm,black] (0,0) grid (5, 5); \end{tikzpicture} & $15/8$ \\
\end{longtabu}
}

\subsection{\texorpdfstring{$5 \times 6$}{}}

{
%% Auto generated by `cgt-cli`
%% Make sure to include preamble from README.md
\begin{longtabu}{m{2.4cm} m{1cm}|m{2.4cm} m{1cm}|m{2.4cm} m{1cm}} 
\hline Position & Temp. & Position & Temp. & Position & Temp. \\ \hline \endhead
\begin{tikzpicture}[scale=0.4] \fill[fill=gray] (2,0) rectangle (3,1); \fill[fill=gray] (3,0) rectangle (4,1); \fill[fill=gray] (5,0) rectangle (6,1); \fill[fill=gray] (0,1) rectangle (1,2); \fill[fill=gray] (4,2) rectangle (5,3); \fill[fill=gray] (5,2) rectangle (6,3); \fill[fill=gray] (0,3) rectangle (1,4); \fill[fill=gray] (1,3) rectangle (2,4); \fill[fill=gray] (5,3) rectangle (6,4); \fill[fill=gray] (0,4) rectangle (1,5); \fill[fill=gray] (1,4) rectangle (2,5); \fill[fill=gray] (3,4) rectangle (4,5); \fill[fill=gray] (5,4) rectangle (6,5); \draw[step=1cm,black] (0,0) grid (6, 5); \end{tikzpicture} & $2$ & \begin{tikzpicture}[scale=0.4] \fill[fill=gray] (2,0) rectangle (3,1); \fill[fill=gray] (3,0) rectangle (4,1); \fill[fill=gray] (5,0) rectangle (6,1); \fill[fill=gray] (0,1) rectangle (1,2); \fill[fill=gray] (4,2) rectangle (5,3); \fill[fill=gray] (0,3) rectangle (1,4); \fill[fill=gray] (1,3) rectangle (2,4); \fill[fill=gray] (5,3) rectangle (6,4); \fill[fill=gray] (0,4) rectangle (1,5); \fill[fill=gray] (1,4) rectangle (2,5); \fill[fill=gray] (3,4) rectangle (4,5); \fill[fill=gray] (5,4) rectangle (6,5); \draw[step=1cm,black] (0,0) grid (6, 5); \end{tikzpicture} & $2$ & \begin{tikzpicture}[scale=0.4] \fill[fill=gray] (2,0) rectangle (3,1); \fill[fill=gray] (3,0) rectangle (4,1); \fill[fill=gray] (0,1) rectangle (1,2); \fill[fill=gray] (4,2) rectangle (5,3); \fill[fill=gray] (0,3) rectangle (1,4); \fill[fill=gray] (1,3) rectangle (2,4); \fill[fill=gray] (5,3) rectangle (6,4); \fill[fill=gray] (0,4) rectangle (1,5); \fill[fill=gray] (1,4) rectangle (2,5); \fill[fill=gray] (3,4) rectangle (4,5); \fill[fill=gray] (5,4) rectangle (6,5); \draw[step=1cm,black] (0,0) grid (6, 5); \end{tikzpicture} & $2$ \\
\begin{tikzpicture}[scale=0.4] \fill[fill=gray] (2,0) rectangle (3,1); \fill[fill=gray] (3,0) rectangle (4,1); \fill[fill=gray] (0,1) rectangle (1,2); \fill[fill=gray] (4,2) rectangle (5,3); \fill[fill=gray] (5,2) rectangle (6,3); \fill[fill=gray] (0,3) rectangle (1,4); \fill[fill=gray] (1,3) rectangle (2,4); \fill[fill=gray] (5,3) rectangle (6,4); \fill[fill=gray] (0,4) rectangle (1,5); \fill[fill=gray] (1,4) rectangle (2,5); \fill[fill=gray] (3,4) rectangle (4,5); \fill[fill=gray] (5,4) rectangle (6,5); \draw[step=1cm,black] (0,0) grid (6, 5); \end{tikzpicture} & $31/16$ & \begin{tikzpicture}[scale=0.4] \fill[fill=gray] (0,0) rectangle (1,1); \fill[fill=gray] (1,0) rectangle (2,1); \fill[fill=gray] (2,0) rectangle (3,1); \fill[fill=gray] (0,1) rectangle (1,2); \fill[fill=gray] (4,1) rectangle (5,2); \fill[fill=gray] (0,2) rectangle (1,3); \fill[fill=gray] (4,2) rectangle (5,3); \fill[fill=gray] (5,2) rectangle (6,3); \fill[fill=gray] (3,3) rectangle (4,4); \fill[fill=gray] (4,3) rectangle (5,4); \fill[fill=gray] (5,3) rectangle (6,4); \fill[fill=gray] (1,4) rectangle (2,5); \fill[fill=gray] (3,4) rectangle (4,5); \fill[fill=gray] (4,4) rectangle (5,5); \fill[fill=gray] (5,4) rectangle (6,5); \draw[step=1cm,black] (0,0) grid (6, 5); \end{tikzpicture} & $15/8$ & \begin{tikzpicture}[scale=0.4] \fill[fill=gray] (0,0) rectangle (1,1); \fill[fill=gray] (1,0) rectangle (2,1); \fill[fill=gray] (2,0) rectangle (3,1); \fill[fill=gray] (4,0) rectangle (5,1); \fill[fill=gray] (1,1) rectangle (2,2); \fill[fill=gray] (2,1) rectangle (3,2); \fill[fill=gray] (5,2) rectangle (6,3); \fill[fill=gray] (0,3) rectangle (1,4); \fill[fill=gray] (1,3) rectangle (2,4); \fill[fill=gray] (5,3) rectangle (6,4); \fill[fill=gray] (0,4) rectangle (1,5); \fill[fill=gray] (1,4) rectangle (2,5); \fill[fill=gray] (3,4) rectangle (4,5); \fill[fill=gray] (4,4) rectangle (5,5); \fill[fill=gray] (5,4) rectangle (6,5); \draw[step=1cm,black] (0,0) grid (6, 5); \end{tikzpicture} & $15/8$ \\
\begin{tikzpicture}[scale=0.4] \fill[fill=gray] (0,0) rectangle (1,1); \fill[fill=gray] (1,0) rectangle (2,1); \fill[fill=gray] (2,0) rectangle (3,1); \fill[fill=gray] (4,0) rectangle (5,1); \fill[fill=gray] (5,0) rectangle (6,1); \fill[fill=gray] (0,1) rectangle (1,2); \fill[fill=gray] (4,1) rectangle (5,2); \fill[fill=gray] (0,2) rectangle (1,3); \fill[fill=gray] (3,3) rectangle (4,4); \fill[fill=gray] (4,3) rectangle (5,4); \fill[fill=gray] (5,3) rectangle (6,4); \fill[fill=gray] (1,4) rectangle (2,5); \fill[fill=gray] (3,4) rectangle (4,5); \fill[fill=gray] (4,4) rectangle (5,5); \fill[fill=gray] (5,4) rectangle (6,5); \draw[step=1cm,black] (0,0) grid (6, 5); \end{tikzpicture} & $15/8$ & \begin{tikzpicture}[scale=0.4] \fill[fill=gray] (0,0) rectangle (1,1); \fill[fill=gray] (1,0) rectangle (2,1); \fill[fill=gray] (2,0) rectangle (3,1); \fill[fill=gray] (4,0) rectangle (5,1); \fill[fill=gray] (1,1) rectangle (2,2); \fill[fill=gray] (2,1) rectangle (3,2); \fill[fill=gray] (5,2) rectangle (6,3); \fill[fill=gray] (1,3) rectangle (2,4); \fill[fill=gray] (5,3) rectangle (6,4); \fill[fill=gray] (0,4) rectangle (1,5); \fill[fill=gray] (1,4) rectangle (2,5); \fill[fill=gray] (3,4) rectangle (4,5); \fill[fill=gray] (4,4) rectangle (5,5); \fill[fill=gray] (5,4) rectangle (6,5); \draw[step=1cm,black] (0,0) grid (6, 5); \end{tikzpicture} & $15/8$ & \begin{tikzpicture}[scale=0.4] \fill[fill=gray] (1,0) rectangle (2,1); \fill[fill=gray] (2,0) rectangle (3,1); \fill[fill=gray] (4,0) rectangle (5,1); \fill[fill=gray] (1,1) rectangle (2,2); \fill[fill=gray] (2,1) rectangle (3,2); \fill[fill=gray] (5,2) rectangle (6,3); \fill[fill=gray] (1,3) rectangle (2,4); \fill[fill=gray] (5,3) rectangle (6,4); \fill[fill=gray] (0,4) rectangle (1,5); \fill[fill=gray] (1,4) rectangle (2,5); \fill[fill=gray] (3,4) rectangle (4,5); \fill[fill=gray] (4,4) rectangle (5,5); \fill[fill=gray] (5,4) rectangle (6,5); \draw[step=1cm,black] (0,0) grid (6, 5); \end{tikzpicture} & $15/8$ \\
\begin{tikzpicture}[scale=0.4] \fill[fill=gray] (2,0) rectangle (3,1); \fill[fill=gray] (4,0) rectangle (5,1); \fill[fill=gray] (1,1) rectangle (2,2); \fill[fill=gray] (2,1) rectangle (3,2); \fill[fill=gray] (5,2) rectangle (6,3); \fill[fill=gray] (1,3) rectangle (2,4); \fill[fill=gray] (5,3) rectangle (6,4); \fill[fill=gray] (0,4) rectangle (1,5); \fill[fill=gray] (1,4) rectangle (2,5); \fill[fill=gray] (3,4) rectangle (4,5); \fill[fill=gray] (4,4) rectangle (5,5); \fill[fill=gray] (5,4) rectangle (6,5); \draw[step=1cm,black] (0,0) grid (6, 5); \end{tikzpicture} & $15/8$ & \begin{tikzpicture}[scale=0.4] \fill[fill=gray] (2,0) rectangle (3,1); \fill[fill=gray] (3,0) rectangle (4,1); \fill[fill=gray] (4,0) rectangle (5,1); \fill[fill=gray] (5,0) rectangle (6,1); \fill[fill=gray] (0,1) rectangle (1,2); \fill[fill=gray] (4,2) rectangle (5,3); \fill[fill=gray] (5,2) rectangle (6,3); \fill[fill=gray] (0,3) rectangle (1,4); \fill[fill=gray] (1,3) rectangle (2,4); \fill[fill=gray] (5,3) rectangle (6,4); \fill[fill=gray] (0,4) rectangle (1,5); \fill[fill=gray] (1,4) rectangle (2,5); \fill[fill=gray] (3,4) rectangle (4,5); \fill[fill=gray] (5,4) rectangle (6,5); \draw[step=1cm,black] (0,0) grid (6, 5); \end{tikzpicture} & $15/8$ & \begin{tikzpicture}[scale=0.4] \fill[fill=gray] (1,0) rectangle (2,1); \fill[fill=gray] (3,0) rectangle (4,1); \fill[fill=gray] (4,0) rectangle (5,1); \fill[fill=gray] (5,0) rectangle (6,1); \fill[fill=gray] (3,1) rectangle (4,2); \fill[fill=gray] (4,1) rectangle (5,2); \fill[fill=gray] (0,2) rectangle (1,3); \fill[fill=gray] (0,3) rectangle (1,4); \fill[fill=gray] (4,3) rectangle (5,4); \fill[fill=gray] (0,4) rectangle (1,5); \fill[fill=gray] (1,4) rectangle (2,5); \fill[fill=gray] (2,4) rectangle (3,5); \fill[fill=gray] (4,4) rectangle (5,5); \draw[step=1cm,black] (0,0) grid (6, 5); \end{tikzpicture} & $15/8$ \\
\begin{tikzpicture}[scale=0.4] \fill[fill=gray] (1,0) rectangle (2,1); \fill[fill=gray] (3,0) rectangle (4,1); \fill[fill=gray] (5,0) rectangle (6,1); \fill[fill=gray] (3,1) rectangle (4,2); \fill[fill=gray] (0,2) rectangle (1,3); \fill[fill=gray] (5,2) rectangle (6,3); \fill[fill=gray] (0,3) rectangle (1,4); \fill[fill=gray] (4,3) rectangle (5,4); \fill[fill=gray] (5,3) rectangle (6,4); \fill[fill=gray] (0,4) rectangle (1,5); \fill[fill=gray] (1,4) rectangle (2,5); \fill[fill=gray] (2,4) rectangle (3,5); \fill[fill=gray] (5,4) rectangle (6,5); \draw[step=1cm,black] (0,0) grid (6, 5); \end{tikzpicture} & $15/8$ & \begin{tikzpicture}[scale=0.4] \fill[fill=gray] (1,0) rectangle (2,1); \fill[fill=gray] (3,0) rectangle (4,1); \fill[fill=gray] (4,0) rectangle (5,1); \fill[fill=gray] (3,1) rectangle (4,2); \fill[fill=gray] (0,2) rectangle (1,3); \fill[fill=gray] (5,2) rectangle (6,3); \fill[fill=gray] (0,3) rectangle (1,4); \fill[fill=gray] (4,3) rectangle (5,4); \fill[fill=gray] (5,3) rectangle (6,4); \fill[fill=gray] (0,4) rectangle (1,5); \fill[fill=gray] (1,4) rectangle (2,5); \fill[fill=gray] (2,4) rectangle (3,5); \fill[fill=gray] (5,4) rectangle (6,5); \draw[step=1cm,black] (0,0) grid (6, 5); \end{tikzpicture} & $15/8$ & \begin{tikzpicture}[scale=0.4] \fill[fill=gray] (1,0) rectangle (2,1); \fill[fill=gray] (3,0) rectangle (4,1); \fill[fill=gray] (3,1) rectangle (4,2); \fill[fill=gray] (0,2) rectangle (1,3); \fill[fill=gray] (5,2) rectangle (6,3); \fill[fill=gray] (0,3) rectangle (1,4); \fill[fill=gray] (4,3) rectangle (5,4); \fill[fill=gray] (5,3) rectangle (6,4); \fill[fill=gray] (0,4) rectangle (1,5); \fill[fill=gray] (1,4) rectangle (2,5); \fill[fill=gray] (2,4) rectangle (3,5); \fill[fill=gray] (5,4) rectangle (6,5); \draw[step=1cm,black] (0,0) grid (6, 5); \end{tikzpicture} & $15/8$ \\
\begin{tikzpicture}[scale=0.4] \fill[fill=gray] (1,0) rectangle (2,1); \fill[fill=gray] (3,0) rectangle (4,1); \fill[fill=gray] (3,1) rectangle (4,2); \fill[fill=gray] (4,1) rectangle (5,2); \fill[fill=gray] (0,2) rectangle (1,3); \fill[fill=gray] (0,3) rectangle (1,4); \fill[fill=gray] (4,3) rectangle (5,4); \fill[fill=gray] (5,3) rectangle (6,4); \fill[fill=gray] (0,4) rectangle (1,5); \fill[fill=gray] (1,4) rectangle (2,5); \fill[fill=gray] (2,4) rectangle (3,5); \fill[fill=gray] (5,4) rectangle (6,5); \draw[step=1cm,black] (0,0) grid (6, 5); \end{tikzpicture} & $15/8$ & \begin{tikzpicture}[scale=0.4] \fill[fill=gray] (3,0) rectangle (4,1); \fill[fill=gray] (4,0) rectangle (5,1); \fill[fill=gray] (5,0) rectangle (6,1); \fill[fill=gray] (1,1) rectangle (2,2); \fill[fill=gray] (5,1) rectangle (6,2); \fill[fill=gray] (1,2) rectangle (2,3); \fill[fill=gray] (5,2) rectangle (6,3); \fill[fill=gray] (1,3) rectangle (2,4); \fill[fill=gray] (2,3) rectangle (3,4); \fill[fill=gray] (0,4) rectangle (1,5); \fill[fill=gray] (1,4) rectangle (2,5); \fill[fill=gray] (2,4) rectangle (3,5); \fill[fill=gray] (4,4) rectangle (5,5); \draw[step=1cm,black] (0,0) grid (6, 5); \end{tikzpicture} & $15/8$ & \begin{tikzpicture}[scale=0.4] \fill[fill=gray] (3,0) rectangle (4,1); \fill[fill=gray] (4,0) rectangle (5,1); \fill[fill=gray] (5,0) rectangle (6,1); \fill[fill=gray] (1,1) rectangle (2,2); \fill[fill=gray] (5,1) rectangle (6,2); \fill[fill=gray] (5,2) rectangle (6,3); \fill[fill=gray] (2,3) rectangle (3,4); \fill[fill=gray] (0,4) rectangle (1,5); \fill[fill=gray] (1,4) rectangle (2,5); \fill[fill=gray] (2,4) rectangle (3,5); \fill[fill=gray] (4,4) rectangle (5,5); \draw[step=1cm,black] (0,0) grid (6, 5); \end{tikzpicture} & $15/8$ \\
\begin{tikzpicture}[scale=0.4] \fill[fill=gray] (1,0) rectangle (2,1); \fill[fill=gray] (2,0) rectangle (3,1); \fill[fill=gray] (4,0) rectangle (5,1); \fill[fill=gray] (2,1) rectangle (3,2); \fill[fill=gray] (5,2) rectangle (6,3); \fill[fill=gray] (1,3) rectangle (2,4); \fill[fill=gray] (5,3) rectangle (6,4); \fill[fill=gray] (0,4) rectangle (1,5); \fill[fill=gray] (1,4) rectangle (2,5); \fill[fill=gray] (3,4) rectangle (4,5); \fill[fill=gray] (4,4) rectangle (5,5); \fill[fill=gray] (5,4) rectangle (6,5); \draw[step=1cm,black] (0,0) grid (6, 5); \end{tikzpicture} & $29/16$ & \begin{tikzpicture}[scale=0.4] \fill[fill=gray] (0,0) rectangle (1,1); \fill[fill=gray] (2,0) rectangle (3,1); \fill[fill=gray] (4,0) rectangle (5,1); \fill[fill=gray] (2,1) rectangle (3,2); \fill[fill=gray] (5,2) rectangle (6,3); \fill[fill=gray] (1,3) rectangle (2,4); \fill[fill=gray] (5,3) rectangle (6,4); \fill[fill=gray] (0,4) rectangle (1,5); \fill[fill=gray] (1,4) rectangle (2,5); \fill[fill=gray] (3,4) rectangle (4,5); \fill[fill=gray] (4,4) rectangle (5,5); \fill[fill=gray] (5,4) rectangle (6,5); \draw[step=1cm,black] (0,0) grid (6, 5); \end{tikzpicture} & $29/16$ & \begin{tikzpicture}[scale=0.4] \fill[fill=gray] (2,0) rectangle (3,1); \fill[fill=gray] (4,0) rectangle (5,1); \fill[fill=gray] (2,1) rectangle (3,2); \fill[fill=gray] (5,2) rectangle (6,3); \fill[fill=gray] (1,3) rectangle (2,4); \fill[fill=gray] (5,3) rectangle (6,4); \fill[fill=gray] (0,4) rectangle (1,5); \fill[fill=gray] (1,4) rectangle (2,5); \fill[fill=gray] (3,4) rectangle (4,5); \fill[fill=gray] (4,4) rectangle (5,5); \fill[fill=gray] (5,4) rectangle (6,5); \draw[step=1cm,black] (0,0) grid (6, 5); \end{tikzpicture} & $29/16$ \\
\begin{tikzpicture}[scale=0.4] \fill[fill=gray] (0,0) rectangle (1,1); \fill[fill=gray] (1,0) rectangle (2,1); \fill[fill=gray] (2,0) rectangle (3,1); \fill[fill=gray] (0,1) rectangle (1,2); \fill[fill=gray] (4,1) rectangle (5,2); \fill[fill=gray] (0,2) rectangle (1,3); \fill[fill=gray] (5,2) rectangle (6,3); \fill[fill=gray] (3,3) rectangle (4,4); \fill[fill=gray] (1,4) rectangle (2,5); \fill[fill=gray] (3,4) rectangle (4,5); \fill[fill=gray] (5,4) rectangle (6,5); \draw[step=1cm,black] (0,0) grid (6, 5); \end{tikzpicture} & $29/16$ & \begin{tikzpicture}[scale=0.4] \fill[fill=gray] (0,0) rectangle (1,1); \fill[fill=gray] (1,0) rectangle (2,1); \fill[fill=gray] (2,0) rectangle (3,1); \fill[fill=gray] (0,1) rectangle (1,2); \fill[fill=gray] (4,1) rectangle (5,2); \fill[fill=gray] (0,2) rectangle (1,3); \fill[fill=gray] (3,3) rectangle (4,4); \fill[fill=gray] (1,4) rectangle (2,5); \fill[fill=gray] (3,4) rectangle (4,5); \fill[fill=gray] (5,4) rectangle (6,5); \draw[step=1cm,black] (0,0) grid (6, 5); \end{tikzpicture} & $29/16$ & \begin{tikzpicture}[scale=0.4] \fill[fill=gray] (1,0) rectangle (2,1); \fill[fill=gray] (2,0) rectangle (3,1); \fill[fill=gray] (4,0) rectangle (5,1); \fill[fill=gray] (2,1) rectangle (3,2); \fill[fill=gray] (0,2) rectangle (1,3); \fill[fill=gray] (5,2) rectangle (6,3); \fill[fill=gray] (1,3) rectangle (2,4); \fill[fill=gray] (5,3) rectangle (6,4); \fill[fill=gray] (3,4) rectangle (4,5); \fill[fill=gray] (4,4) rectangle (5,5); \fill[fill=gray] (5,4) rectangle (6,5); \draw[step=1cm,black] (0,0) grid (6, 5); \end{tikzpicture} & $29/16$ \\
\end{longtabu}
}

\subsection{\texorpdfstring{$6 \times 6$}{}}

Due to resource limitations, the $6 \times 6$ search results are not exhaustive. The search space was limited to positions with no more than 20 empty tiles. The table includes unique, up to rotation and reflection, positions with temperature $2$.

{
%% Auto generated by `cgt-cli`
%% Make sure to include preamble from README.md
\begin{longtabu}{m{2.4cm} m{1cm}|m{2.4cm} m{1cm}|m{2.4cm} m{1cm}} 
\hline Position & Temp. & Position & Temp. & Position & Temp. \\ \hline \endhead
\begin{tikzpicture}[scale=0.4] \fill[fill=gray] (0,0) rectangle (1,1); \fill[fill=gray] (2,0) rectangle (3,1); \fill[fill=gray] (4,0) rectangle (5,1); \fill[fill=gray] (5,0) rectangle (6,1); \fill[fill=gray] (0,1) rectangle (1,2); \fill[fill=gray] (4,1) rectangle (5,2); \fill[fill=gray] (5,1) rectangle (6,2); \fill[fill=gray] (1,2) rectangle (2,3); \fill[fill=gray] (5,3) rectangle (6,4); \fill[fill=gray] (0,4) rectangle (1,5); \fill[fill=gray] (2,4) rectangle (3,5); \fill[fill=gray] (3,4) rectangle (4,5); \fill[fill=gray] (0,5) rectangle (1,6); \fill[fill=gray] (2,5) rectangle (3,6); \fill[fill=gray] (3,5) rectangle (4,6); \fill[fill=gray] (4,5) rectangle (5,6); \fill[fill=gray] (5,5) rectangle (6,6); \draw[step=1cm,black] (0,0) grid (6, 6); \end{tikzpicture} & $2$ & \begin{tikzpicture}[scale=0.4] \fill[fill=gray] (0,0) rectangle (1,1); \fill[fill=gray] (1,0) rectangle (2,1); \fill[fill=gray] (3,0) rectangle (4,1); \fill[fill=gray] (5,0) rectangle (6,1); \fill[fill=gray] (0,1) rectangle (1,2); \fill[fill=gray] (1,1) rectangle (2,2); \fill[fill=gray] (5,1) rectangle (6,2); \fill[fill=gray] (4,2) rectangle (5,3); \fill[fill=gray] (0,3) rectangle (1,4); \fill[fill=gray] (2,4) rectangle (3,5); \fill[fill=gray] (3,4) rectangle (4,5); \fill[fill=gray] (5,4) rectangle (6,5); \fill[fill=gray] (0,5) rectangle (1,6); \fill[fill=gray] (1,5) rectangle (2,6); \fill[fill=gray] (2,5) rectangle (3,6); \fill[fill=gray] (3,5) rectangle (4,6); \draw[step=1cm,black] (0,0) grid (6, 6); \end{tikzpicture} & $2$ & \begin{tikzpicture}[scale=0.4] \fill[fill=gray] (0,0) rectangle (1,1); \fill[fill=gray] (2,0) rectangle (3,1); \fill[fill=gray] (4,0) rectangle (5,1); \fill[fill=gray] (5,0) rectangle (6,1); \fill[fill=gray] (0,1) rectangle (1,2); \fill[fill=gray] (4,1) rectangle (5,2); \fill[fill=gray] (5,1) rectangle (6,2); \fill[fill=gray] (1,2) rectangle (2,3); \fill[fill=gray] (5,3) rectangle (6,4); \fill[fill=gray] (0,4) rectangle (1,5); \fill[fill=gray] (2,4) rectangle (3,5); \fill[fill=gray] (3,4) rectangle (4,5); \fill[fill=gray] (0,5) rectangle (1,6); \fill[fill=gray] (3,5) rectangle (4,6); \fill[fill=gray] (4,5) rectangle (5,6); \fill[fill=gray] (5,5) rectangle (6,6); \draw[step=1cm,black] (0,0) grid (6, 6); \end{tikzpicture} & $2$ \\
\begin{tikzpicture}[scale=0.4] \fill[fill=gray] (0,0) rectangle (1,1); \fill[fill=gray] (1,0) rectangle (2,1); \fill[fill=gray] (3,0) rectangle (4,1); \fill[fill=gray] (5,0) rectangle (6,1); \fill[fill=gray] (0,1) rectangle (1,2); \fill[fill=gray] (1,1) rectangle (2,2); \fill[fill=gray] (5,1) rectangle (6,2); \fill[fill=gray] (4,2) rectangle (5,3); \fill[fill=gray] (5,2) rectangle (6,3); \fill[fill=gray] (0,3) rectangle (1,4); \fill[fill=gray] (4,3) rectangle (5,4); \fill[fill=gray] (5,3) rectangle (6,4); \fill[fill=gray] (2,4) rectangle (3,5); \fill[fill=gray] (0,5) rectangle (1,6); \fill[fill=gray] (1,5) rectangle (2,6); \fill[fill=gray] (4,5) rectangle (5,6); \draw[step=1cm,black] (0,0) grid (6, 6); \end{tikzpicture} & $2$ & \begin{tikzpicture}[scale=0.4] \fill[fill=gray] (0,0) rectangle (1,1); \fill[fill=gray] (1,0) rectangle (2,1); \fill[fill=gray] (4,0) rectangle (5,1); \fill[fill=gray] (5,0) rectangle (6,1); \fill[fill=gray] (2,1) rectangle (3,2); \fill[fill=gray] (0,2) rectangle (1,3); \fill[fill=gray] (4,2) rectangle (5,3); \fill[fill=gray] (4,3) rectangle (5,4); \fill[fill=gray] (5,3) rectangle (6,4); \fill[fill=gray] (0,4) rectangle (1,5); \fill[fill=gray] (1,4) rectangle (2,5); \fill[fill=gray] (5,4) rectangle (6,5); \fill[fill=gray] (0,5) rectangle (1,6); \fill[fill=gray] (1,5) rectangle (2,6); \fill[fill=gray] (3,5) rectangle (4,6); \fill[fill=gray] (5,5) rectangle (6,6); \draw[step=1cm,black] (0,0) grid (6, 6); \end{tikzpicture} & $2$ \\
\end{longtabu}
}

\subsection{\texorpdfstring{$7 \times 7$}{}}
The $7 \times 7$ positions listed below were found using a simple genetic algorithm that started with a seed of manually discovered $7 \times 7$ grids with temperature $2$ and mutated it while saving positions that maintained a temperature of $2$.

\foreach \i in {1,...,51}{
    \hspace*{-1cm}
    % order for trim = left, bottom, right, top
    \includegraphics[scale=1, page=\i, trim={1.75in 1.85in 1.25in 1.5in}, clip]{out7x7-genetic-2.gen.pdf}
    \newpage
}

\appendix
\section{Running Computations Using cgt-tools}

To run an exhaustive search or to generate \LaTeX{} tables from search results it is required to install \href{https://github.com/t4ccer/cgt-tools}{\texttt{cgt-tools}}, that was developed while writing this paper.
It is written in Rust, which can be installed following the \href{https://www.rust-lang.org/tools/install}{official guide}. The code of \texttt{cgt-tools} is published on GitHub so \href{https://git-scm.com/downloads}{\texttt{git}} is also required to obtain the source code.

Below the required commands to run the searches are listed.

\subsection{Compiling \texttt{cgt-tools}}

\begin{lstlisting}[frame=single]
$ git clone https://github.com/t4ccer/cgt-tools
$ cd cgt-tools
# Version 0.5.1 was used while writing this paper
# Some commands may change in the next versions
$ git checkout v0.5.1
$ cargo install --path cgt_cli
# Verify that installation was successful
$ cgt-cli --help
\end{lstlisting}

\subsection{Running exhaustive search}

\begin{lstlisting}[frame=single]
# Set <width> and <height> to desired grid size
# Set <search-output> to the desired file path with search results
$ cgt-cli domineering exhaustive-search \
        --width <width> --height <height> \
        --output-path <search-output>
\end{lstlisting}

Exhaustive search supports various configurable options like restricting search space to a given number of empty tiles, including or excluding decompositions, switching thermograph calculation methods and many more. They can all be viewed with \verb|cgt-cli domineering exhaustive-search --help|.

\subsection{Generating tables with search results}

\begin{lstlisting}[frame=single]
$ cgt-cli domineering latex-table \
        --in-file <search-output> --out-file my-table.tex
\end{lstlisting}

It is possible to tweak the number of columns and figure sizes. See all possible options with \verb|cgt-cli domineering latex-table --help|. Generated latex file can be later used with \verb|\input{<path>}| command in latex documents.

\begin{lstlisting}[frame=single]
\input{my-table}.tex
\end{lstlisting}

Note that the following latex packages must be imported to include the generated table.

\begin{lstlisting}[frame=single]
\usepackage{tabu}
\usepackage{tikz}
\usepackage{longtable} \tabulinesep=1.2mm
\end{lstlisting}

% \bibliography{citations}{}
% \bibliographystyle{plain}

\end{document}